\documentclass[11pt,dvips,twoside,letterpaper]{article}
\usepackage{pslatex}
\usepackage{fancyhdr}
\usepackage{graphicx}
\usepackage{geometry}
\RequirePackage[latin1]{inputenc} \RequirePackage[T1]{fontenc}

\def\figurename{Figure} % Replace the colon that normally appears after the Figure number by a period.
\makeatletter
\renewcommand{\fnum@figure}[1]{\figurename~\thefigure.}
\makeatother

\def\tablename{Table} % Replace the colon that normally appears after the Figure number by a period.
\makeatletter
\renewcommand{\fnum@table}[1]{\tablename~\thetable.}
\makeatother
\usepackage{color}
\usepackage[colorlinks]{hyperref}
                [2000/03/29 v1.0m Standard LaTeX package]

\usepackage{amsmath}
\usepackage{amssymb}
\usepackage{amsfonts}
\usepackage{amsthm,amscd}
\newtheorem{theorem}{Theorem}[section]
\newtheorem{lemma}[theorem]{Lemma}

\theoremstyle{definition}

\newtheorem{definition}[theorem]{Definition}

\theoremstyle{remark}

\numberwithin{equation}{section}

\def\E{\mathbb E}

\def\E{\mathbb E}

\begin{document}

%------------------------------------------------------------------------------
\title{Backward doubly stochastic integral equations of the Volterra
type}

\author{Jean-Marc OWO \vspace{0.2cm}\\
Universit\'{e} de Cocody\\ UFR de Math\'{e}matiques et Informatique
\\ $22$ BP $582$ Abidjan $22$, C\^{o}te d'Ivoire}%

%\address{
%Universit\'{e} de Cocody-Abidjan, UFR de Math\'{e}matiques et
%Informatique, \newline \indent Laboratoire de Probabilités et
%Statistique: 22 BP 582 Abidjan 22, C\^{o}te d'Ivoire}

%\email{owojmarc@hotmail.com}

\date{}
\maketitle \thispagestyle{empty} \setcounter{page}{1}

% ------- [First Page Running Head] - place it immediately after title! ------
\thispagestyle{fancy} \fancyhead{}
 \fancyfoot{}
\renewcommand{\headrulewidth}{0pt}
%------------------------------------

\begin{abstract}
%EndExpansion
In this paper, we study backward doubly stochastic integral
equations of the Volterra type ( BDSIEVs in short). Under uniform
Lipschitz assumptions, We establish an existence and uniqueness
result.
%TCIMACRO{\TeXButton{End abstract}{\end{abstract}}}%
%BeginExpansion
\end{abstract}%
%EndExpansion

%TCIMACRO{\TeXButton{Begin keywords}{\begin{keyword}} }%
%BeginExpansion

\vspace{.08in} \noindent \textbf{Keywords.} Volterra integrals,
backward stochastic integral, backward doubly stochastic Volterra
integral equations.

\vspace{.08in} \noindent \textbf{Mathematics Subject Classification
(2000).} 60H05, 60H20, 60H35

\section{Introduction}
Backward doubly stochastic differential equations (BDSDEs for short)
are equations with two different directions of stochastic integrals,
i.e., the equations involve both a standard (forward) stochastic
integral $dW_{t}$ and a backward stochastic integral
$\overleftarrow{dB_t}$:
\begin{eqnarray}\label{v}
Y(t)=\xi+\int_{t}^{T}f(s,Y(s),Z(s))ds+
\int_{t}^{T}g(s,Y(s),Z(s))\overleftarrow{dB_s}-\int_{t}^{T}Z(s)dW_{s}.
\end{eqnarray}
This kind of equation was introduced by Pardoux and Peng \cite{Pard.
Peng} in $1994$. They proved the existence and uniqueness of
solutions for BDSDEs under uniform Lipschitz conditions. Many others
investigations concerned BDSDEs were made with weaker conditions
namely by Zhou and al. \cite{Z.al} in $2004$ with non-Lipschitz
assumptions witch in turn were weakened by Han Baoyan and al.
\cite{Han} in $2005$ and recently by N'zi and Owo \cite{owo} (2008).
In \cite{Shi} ($2005$), Shi and al. weaken the uniform Lipschitz
assumptions to linear growth and continuous conditions by virtue of
the comparison theorem that is introduced by themselves. They obtain
the existence of solutions to BDSDE but without uniqueness. Pursuing
their investigations on BDSDEs, N'zi and Owo \cite{owo1} (2009)
obtained recently an existence result with discontinuous conditions.
Meanwhile, an other line of researches concerned with backward
stochastic integral equations of Volterra type (BSIEVs for short)
i.e., equations in form:
\begin{eqnarray}\label{v1}
Y(t)=\xi+\int_{t}^{T}f(t,s,Y(s),Z(t,s))ds- \int_{t}^{T}Z(t,s)dW_{s},
\end{eqnarray}
are lead by Lin \cite{lin} in $(2002)$ under global Lipschitz
condition on the drift witch was recently weakened by A. Aman and
M. N'zi \cite{Aman} in ($2005$) to local Lipschitz condition.\\
Recently, general case of BSIEVs \eqref{v1}, has been studied by J.
Yong [\cite{Yong},\cite{Yong1}] ($2006$).

\medskip
The purpose of this paper is to generalize the theory of Volterra
equations to backward doubly stochastic integral equations.
\\
Thus, we consider the following equation:
\begin{eqnarray}\label{v2}
Y(t)=\xi+\int_{t}^{T}f(t,s,Y(s),Z(t,s))ds+
\int_{t}^{T}g(t,s,Y(s),Z(t,s))\overleftarrow{dB_s}
-\int_{t}^{T}Z(t,s)dW_{s},
\end{eqnarray}
that we call backward doubly stochastic Volterra integral equations
(in short BDSVIEs).

\medskip
Note that when  $f$ and $g$ do not depend on $t$, backward doubly
stochastic Volterra integral equations (BDSVIEs) coincide with
backward doubly stochastic differential equations (BDSDEs) .

\medskip
\noindent The present paper is organized as follows : in section
$2$, we deal with notations and set up our framework assumptions and
give the definition of adapted solutions to BDSVIEs. The section $3$
is concerned with the main result.
\section{Preliminaries}
\subsection{Notations}
The Euclidean norm of a vector $x \in \mathbb{R}^{k}$ will be denote
by  $|x|$, and for an element $z\in \mathbb{R}^{d \times k}$
considered as a $d \times k$ matrix, we define its Euclidean norm by
$||y ||=\sqrt{Tr(zz^{*})}$ and $
 <z,y>=Tr(zy^{*}),$ where $y^{*}$ is
the transpose of $y$.\\
Let $(\Omega, \mathcal{F},\mathbb{P})$ be a probability space and
$T$ be a fixed final time.  Throughout this paper $\{W_{t}; 0\leq
t\leq T \}$ and $\{B_{t}; 0\leq t\leq T \}$ will denote two mutually
independent standard Brownian motion processes, with values
$\mathbb{R}^{d}$ and $\mathbb{R}^{l}$, respectively, defined on
$(\Omega, \mathcal{F},\mathbb{P})$.\\ Let $\mathcal{N}$ denote the
class of $P$-null sets of $\mathcal{F}$. For each $(t,s) \in
[0,T]^2$, we define
$$\mathcal{F}_{t,s}=\mathcal{F}_{t}^{W} \vee \mathcal{F}_{s,T}^{B},\ \ \mathcal{F}_{t}=\mathcal{F}_{t,t} \ \text{and}\ \
\mathbb{F} =\{\mathcal{F}_{t}\}_{t\geq 0},$$ where for any process
$\{x_{t}\}$ ; $\mathcal{F}_{u,t}^{x}=\sigma \{x_{r}-x_{u}; u\leq r
\leq t \} \vee \mathcal{N}$,
$\mathcal{F}_{t}^{x}=\mathcal{F}_{0,t}^{x}$.

\medskip
\noindent Also, we set $\mathbb{F}_{\cdot
s}=\{\mathcal{F}_{t,s}\}_{t\geq 0}$ and
$\mathbb{F}_{t\cdot}=\{\mathcal{F}_{t,s}\}_{s\geq 0}$

\medskip
For $S\in [0,T]$, set $\mathcal{D}_{S,T}=[S,T]^2$;
$\mathcal{D}=\mathcal{D}_{0,T}$ and denote by $\mathcal{P}$ the
$\sigma$-algebra of $\mathcal{F}_{T}$-measurable subsets of $\Omega
\times \mathcal{D}$.

\noindent For any $n \in \mathbb{N}$, let
$\mathcal{M}^{2}(S,T,\mathbb{R}^{n})$ denote the set of (class of
$dP\otimes dt$ a.e. equal) $n-$dimensional jointly measurable random
processes $\varphi:\Omega \times [S,T]\rightarrow \mathbb{R}^{n}$
which satisfy:

\begin{enumerate}
\item[(i)] $|| \varphi ||_{\mathcal{M}^{2}(S,T)}^{2}=\displaystyle\mathbb{E}%
\left(\int_{S}^{T}\mid \varphi(t) \mid^{2} dt\right)< \infty$

\item[(ii)] $\varphi(t)$ is $\mathcal{F}_{t}-$measurable, for a.e. $t \in
[S,T].$
\end{enumerate}

\noindent Similarly, we denote by
$\mathcal{M}^{2}(\mathcal{D}_{S,T},\mathbb{R}^{n})$ the set of
(class of $dP\otimes ds\otimes dt$ a.e. equal) $n-$dimensional
jointly measurable random processes $\psi:\Omega \times
\mathcal{D}_{S,T}\rightarrow \mathbb{R}^{n}$ satisfying:

\begin{enumerate}
\item[(i)] $|| \psi ||_{\mathcal{M}^{2}(\mathcal{D}_{S,T})}^{2}=\displaystyle\mathbb{E}%
\left(\int_{S}^{T}\int_{S}^{T}| \psi(t,s) |^{2} dsdt\right)< \infty$

\item[(ii)] $\psi(t,s)$ is $\mathcal{F}_{s}-$measurable, for a.e. $s \in
[S,T],$ and any $t \in [S,T].$
\end{enumerate}

\noindent Finally, we set
$$\mathcal{H}^{2}(\mathcal{D}_{S,T})=\mathcal{M}^{2}(S,T,\mathbb{R}^{k})\times
\mathcal{M}^{2}(\mathcal{D}_{S,T},\mathbb{R}^{k \times d}),$$ with
the norm
$$||(y(.),z(.,.))||_{\mathcal{H}^{2}(\mathcal{D}_{S,T})}^{2}=\mathbb{E}\left\{
\int_{S}^{T}|y(t)|^2dt+
\int_{S}^{T}\int_{S}^{T}|z(t,s)|^2dsdt\right\}< \infty.$$
Furthermore, let $\mathrm{L}^{2}(\Omega, \mathcal{F}_{T},
\mathbb{P}, \mathbb{R}^k)$ be the set of $k$-dimensional
$\mathcal{F}_{T}$-measurable random vector $\xi$ such that
$$\E\left(|\xi|^2\right)<\infty.$$ We will denote by
$\mathcal{B}_{k}$ the Borel $\sigma$-algebra of $\mathbb{R}^{k}$.
\subsection{Assumptions and definition}
Let $$f:\Omega \times \mathcal{D}\times \mathbb{R}^{k} \times
\mathbb{R}^{k \times d}\rightarrow \mathbb{R}^{k}\ \ \text{and}\ \
g:\Omega \times \mathcal{D}\times \mathbb{R}^{k}\times \mathbb{R}^{k
\times d} \rightarrow \mathbb{R}^{k \times l}$$ be
$(\mathcal{P}\otimes \mathcal{B}_{k}\otimes \mathcal{B}_{k \times
d}\ / \ \mathcal{B}_{k})$ resp. $(\mathcal{P}\otimes
\mathcal{B}_{k}\otimes \mathcal{B}_{k \times d}\ / \ \mathcal{B}_{k
\times l})$ -measurable functions such that for any $(y,z)\in
\mathbb{R}^{k}\times \mathbb{R}^{k \times d},$

\noindent$(H_{1})$
$$ f(.,.,y,z) \in
\mathcal{M}^{2}(\mathcal{D},\mathbb{R}^{k})\ \text{and}\ g(.,.,y,z)
\in \mathcal{M}^{2}(\mathcal{D},\mathbb{R}^{k\times l}).$$ We assume
moreover that there exist constants $C>0$ and $0<\alpha<1$ such that
for any $(\omega, (t,s)) \in \Omega \times \mathcal{D}$ and any
$(y_{1},z_{1}), (y_{2}, z_{2}) \in \mathbb{R}^{k} \times
\mathbb{R}^{k \times d},$

\noindent$(H_{2})$
\begin{eqnarray*}
\mid f(\omega,t,s,y_{1},z_{1})-f(\omega,t,s,y_{2},z_{2}) \mid^{2}
\leq C(\mid y_{1}-y_{2}\mid^{2}+\mid \mid z_{1}-z_{2}\mid \mid^{2})
\end{eqnarray*}
\begin{eqnarray*}
\mid g(\omega,t,s,y_{1},z_{1})-g(\omega,t,s,y_{2},z_{2}) \mid^{2}
\leq C\mid y_{1}-y_{2}\mid^{2}+\alpha\mid \mid z_{1}-z_{2}\mid
\mid^{2}
\end{eqnarray*}
Furthermore, let

\noindent$(H_{3})$ $$\xi \in \mathrm{L}^{2}(\Omega, \mathcal{F}_{T},
\mathbb{P}, \mathbb{R}^k).\hspace{4cm}$$

\medskip
Now, we consider the following BDSVIE: $0\leq t \leq T$

\begin{eqnarray}\label{v2a}
Y(t)=\xi+\int_{t}^{T}f(t,s,Y(s),Z(t,s))ds+
\int_{t}^{T}g(t,s,Y(s),Z(t,s))\overleftarrow{dB_s}
-\int_{t}^{T}Z(t,s)dW_{s}.
\end{eqnarray}
\begin{definition}
A pair of processes $(Y(.),Z(.,.))$ where $Y:\Omega
\times[0,T]\rightarrow \mathbb{R}^{k}$ and $Z:\Omega \times
\mathcal{D} \rightarrow \mathbb{R}^{k \times d}$ is called adapted
solution of \eqref{v2a} if $(Y(.),Z(.,.))\in
\mathcal{H}^{2}(\mathcal{D})$ and satisfies \eqref{v2a}.
\end{definition}

\section{Existence and uniqueness of the adapted solution to BDSVIE}
\noindent To reach the main result, we consider first the equation
\eqref{v2a} where $f$ and $g$ do not depend on $y$ and $z$. That is
\begin{eqnarray}\label{eq21}
Y(t)=\xi+\int_{t}^{T}f(t,s)ds+
\int_{t}^{T}g(t,s)\overleftarrow{dB_s}-\int_{t}^{T}Z(t,s)dW_{s},\
t\in[0,T]
\end{eqnarray}
\begin{lemma}\label{l1}
Let $(H_{1})$, $(H_{2})$ and $(H_{3})$ hold. Then, BDSVIE
\eqref{eq21} admits a unique adapted solution$(Y(.),Z(.,.))\in
\mathcal{H}^{2}(\mathcal{D})$.

\medskip
\noindent
% with the following
%relation:
%\begin{eqnarray*}\label{eq22}
%Y(t)=\E\left(\xi+\int_{t}^{T}f(t,s)ds+
%\int_{t}^{T}g(t,s)\overleftarrow{dB_s}\ |\ %\mathcal{F}_{t,T}^{B}\right)+\int_{0}^{t}Z(t,s)dW_{s},\ t\in[0,T]
%\end{eqnarray*}
Moreover, the following estimate holds:
\begin{eqnarray}\label{eq23}
&&\mathbb{E}\int_{S}^{T}|Y(t)|^2dt+
\mathbb{E}\int_{S}^{T}\int_{S}^{T}|Z(t,s)|^2dsdt
\\
&\leq& 9(T-S)\mathbb{E}|\xi|^2+9\big[(T-S)\vee
1\big]\mathbb{E}\int_{S}^{T}\int_{S}^{T}\Big(|f(t,s)|^2+|g(t,s)|^2\Big)dsdt,
\nonumber
\end{eqnarray}for any $S\in [0,T]$.
\end{lemma}
\begin{proof}
For any $t\in[0,T]$, consider the process $M(t,.)$ defined by: $r\in
[0,T]$, $$M(t,r)=\E\left(\xi+\int_{0}^{T}f(t,s)ds+
\int_{0}^{T}g(t,s)\overleftarrow{dB_s}\ |\
\mathcal{F}_{r,0}\right).$$ The process  $M(t,.)$ as defined, is
$\mathbb{F}_{\cdot 0}-$square integrable martingale. Therefore, by
the extension of Itô's martingale representation theorem, there
exists a $\mathcal{F}_{r,0}-$progressively measurable process
$Z(t,.)$ with valued in $\mathbb{R}^{k \times d}$ such that
$$\int_{0}^{T}|Z(t,s)|^2ds<\infty$$ and
\begin{eqnarray*}
M(t,r)=M(t,0)+\int_{0}^{r}Z(t,s)dW_s, \ \ \forall\ r\in [0,T].
\end{eqnarray*}Hence,
\begin{eqnarray}\label{eq25a}
M(t,r)=M(t,T)-\int_{r}^{T}Z(t,s)dW_s, \ \ \forall\ r\in [0,T].
\end{eqnarray}
%$M_{t}^{t}=M_{0}^{t}+\displaystyle\int_{0}^{t}Z(t,s)dW_s$.% and \ $M_{T}^{t}=M_{0}^{t}+\displaystyle\int_{0}^{T}Z(t,s)dW_s$.

\medskip
\noindent By definition,
$M(t,T)=\xi+\displaystyle\int_{0}^{T}f(t,s)ds+
\int_{0}^{T}g(t,s)\overleftarrow{dB_s}$ and
\begin{eqnarray*}
M(t,r)=N(t,r)+\int_{0}^{r}f(t,s)ds+
\int_{0}^{r}g(t,s)\overleftarrow{dB_s},
\end{eqnarray*} where $$N(t,r)= \E\left(\xi+\int_{r}^{T}f(t,s)ds+
\int_{r}^{T}g(t,s)\overleftarrow{dB_s}\ |\
\mathcal{F}_{r,0}\right).$$

\noindent Therefore
\begin{eqnarray}\label{eq25b}
N(t,r)=\xi+\int_{r}^{T}f(t,s)ds+
\int_{r}^{T}g(t,s)\overleftarrow{dB_s}-\int_{r}^{T}Z(t,s)dW_s.
\end{eqnarray}

Note that $\mathcal{F}_{r,0}=\mathcal{F}_{r}^{W} \vee
\mathcal{F}_{0,T}^{B}=\mathcal{F}_{r} \vee \mathcal{F}_{r}^{B}.$
Then

\begin{eqnarray*}
N(t,r)= \E\left(\theta(\xi,t,r,T)\ |\ \mathcal{F}_{r} \vee
\mathcal{F}_{r}^{B}\right),
\end{eqnarray*}

\noindent where
$\theta(\xi,t,r,T)=\xi+\displaystyle\int_{r}^{T}f(t,s)ds+
\int_{r}^{T}g(t,s)\overleftarrow{dB_s}$ is $\mathcal{F}_{T}^{W} \vee
\mathcal{F}_{r,T}^{B}-$measurable. Consequently,
$\mathcal{F}_{r}^{B}$ is independent of $\mathcal{F}_{r} \vee
\sigma(\theta(\xi,t,r,T))$ and
\begin{eqnarray*}
N(t,r)= \E\left(\theta(\xi,t,r,T)\ |\ \mathcal{F}_{r}\right).
\end{eqnarray*}

\medskip
\noindent Define
\begin{eqnarray}\label{eq25}
Y(t)=N(t,t)=\E\left(\xi+\int_{t}^{T}f(t,s)ds+
\int_{t}^{T}g(t,s)\overleftarrow{dB_s}\ |\ \mathcal{F}_{t}\right).
\end{eqnarray}
Obviously, $Y(.)$ is $\mathbb{F}-$adapted and satisfies the
following relation:
\begin{eqnarray*}
Y(t)=\xi+\int_{t}^{T}f(t,s)ds+
\int_{t}^{T}g(t,s)\overleftarrow{dB_s}-\int_{t}^{T}Z(t,s)dW_{s},\
t\in[0,T].
\end{eqnarray*}
%\medskip
%Now, let us consider the random function  $(t,s)\in \mathcal{D}\rightarrow Z(t,s)$.
%
%\medskip
%\noindent First, , for any  $t\in[0,T],$ $Z(t,s)$ is $\mathcal{F}_{s}^{W} \vee %\mathcal{F}_{t,T}^{B}-$measurable for any $s\in[0,T]$. Consequently, for $s\leq t$, %$Z(t,s)$ is $\mathcal{F}_{s}-$measurable.

\medskip
\noindent On the other hand, we know from above that, for any
$t\in[0,T],$ $Z(t,.)$ is $\mathbb{F}_{\cdot 0}-$adapted and
satisfies equation \eqref{eq25b}, so
\begin{eqnarray}\label{R}
\int_{r}^{T}Z(t,s)dW_s=\theta(\xi,t,r,T)-N(t,r), \ r\in [0,T].
\end{eqnarray}

\medskip
\noindent Since,  $\theta(\xi,t,r,T)$ (resp. $N(t,r)$) is
$\mathcal{F}_{T}^{W} \vee \mathcal{F}_{r,T}^{B}-$ (resp.
$\mathcal{F}_{r}^{W} \vee \mathcal{F}_{r,T}^{B}-$) measurable, it
follows that $\displaystyle\int_{r}^{T}Z(t,s)dW_s$ is
$\mathcal{F}_{T}^{W} \vee \mathcal{F}_{r,T}^{B}-$measurable, for any
$r\in [0,T]$.

\medskip
\noindent Hence, from Itô's martingale representation theorem,
$\{Z(t,s), r<s<T\}$ is  $\mathbb{F}_{\cdot r}-$adapted.
Consequently, $Z(t,s)$ is $\mathcal{F}_{s}^{W} \vee
\mathcal{F}_{r,T}^{B}$-measurable, for any  $r<s<T$. It follows that
$Z(t,s)$ is $\underset{r<s}{\bigwedge}(\mathcal{F}_{s}^{W} \vee
\mathcal{F}_{r,T}^{B})-$measurable. \\But,
$$\underset{r<s}{\bigwedge}(\mathcal{F}_{s}^{W} \vee
\mathcal{F}_{r,T}^{B})=\mathcal{F}_{s}^{W} \vee (
\underset{r<s}{\bigwedge}\mathcal{F}_{r,T}^{B})=\mathcal{F}_{s}^{W}
\vee\mathcal{F}_{s,T}^{B}.$$ Thus, $Z(t,s)$ is $\mathcal{F}_{s}^{W}
\vee \mathcal{F}_{s,T}^{B}-$measurable, for any $(t,s)\in
\mathcal{D}$.

\medskip
Now, let us prove that $(Y(.),Z(.,.))\in
\mathcal{H}^{2}(\mathcal{D})$.

\medskip
\noindent To this end, we use the relations \eqref{eq25} and
\eqref{R}. Then, we obtain:
\begin{eqnarray*}
\E\int_{0}^{T}|Y(t)|^2dt\leq3\E\left(T|\xi|^2+\int_{0}^{T}\int_{0}^{T}\left(T\left|f(t,s)\right|^2+
\left|g(t,s)\right|^2\right)dsdt\right).
\end{eqnarray*}
and
\begin{eqnarray*}
\E\int_{0}^{T}\int_{0}^{T}\left|Z(t,s)\right|^2dsdt\leq6\E\left(T|\xi|^2+\int_{0}^{T}\int_{0}^{T}\left(T\left|f(t,s)\right|^2+
\left|g(t,s)\right|^2\right)dsdt\right).
\end{eqnarray*}
Hence, it follows from conditions $(H_{1})$-$(H_{3})$, that
$(Y(.),Z(.,.))\in \mathcal{H}^{2}(\mathcal{D})$.

\medskip
\noindent Therefore, $(Y(.),Z(.,.))$ is an adapted solution to
\eqref{v2a}.

\medskip
For the uniqueness, let us suppose that $(Y'(.),Z'(.,.))\in
\mathcal{H}^{2}(\mathcal{D})$ is an other adapted solution. Then we
have, $\forall\  t \in [0,T]$
\begin{eqnarray}\label{0}
Y(t)-Y'(t)+\int_{t}^{T}\big[Z(t,s)-Z'(t,s)\big]dW_s=0.
\end{eqnarray}
Taking $\mathbb{E}[.|\mathcal{F}_t]$ in \eqref{0}, we get \
$Y(t)-Y'(t)=0$, $\forall t \in [0,T]$.

\medskip
\noindent On the other hand, still using \eqref{0}, we deduce
$$\mathbb{E}\int_{0}^{T}\int_{0}^{T}|Z(t,s)-Z'(t,s)|^2dsdt=\mathbb{E}\int_{0}^{T}|Y(t)-Y'(t)|^2dt=0.$$
\end{proof}
Our main result in this paper is the following Theorem.
\begin{theorem}\label{t2}
Let $(H_{1})$, $(H_{2})$ and $(H_{3})$ hold. Then the BDSVIE
\eqref{v2a} admits a unique adapted solution $(Y(.),Z(.,.))\in
\mathcal{H}^{2}(\mathcal{D})$.
\end{theorem}

\begin{proof}
For simplicity of notations, we write $Y$ for $Y(.)$ and $Z$ for
$Z(.,.)$, throughout the proof.

For any $(y,z)\in \mathcal{H}^{2}(\mathcal{D}_{S,T})$, we consider
the following BDSVIE: $t\in[S,T]$
\begin{eqnarray}\label{eq5}
Y(t)=\xi+\int_{t}^{T}f(t,s,y(s),z(t,s))ds+
\int_{t}^{T}g(t,s,y(s),z(t,s))\overleftarrow{dB_s}
-\int_{t}^{T}Z(t,s)dW_{s}.
\end{eqnarray}
Thus, by lemma \ref{l1}, Eq. \eqref{eq5} admits a unique adapted
solution $(Y,Z)\in \mathcal{H}^{2}(\mathcal{D}_{S,T})$ and
\begin{eqnarray*}
&&\mathbb{E}\int_{S}^{T}|Y(t)|^2dt+
\mathbb{E}\int_{S}^{T}\int_{S}^{T}|Z(t,s)|^2dsdt
\\
&\leq& 9(T-S)\mathbb{E}|\xi|^2+9\big[(T-S)\vee
1\big]\mathbb{E}\int_{S}^{T}\int_{S}^{T}\Big(|f(t,s,0,0)|^2+|g(t,s,0,0)|^2\Big)dsdt\\
&+&9C(T-S)\big[(T-S)+1\big]\mathbb{E}\int_{S}^{T}|y(t)|^2dt+
9\big[(T-S)C+\alpha\big]\mathbb{E}\int_{S}^{T}\int_{S}^{T}|z(t,s)|^2dsdt
\\
&\leq& 9(T-S)\mathbb{E}|\xi|^2+9\big[(T-S)\vee
1\big]\mathbb{E}\int_{S}^{T}\int_{S}^{T}\Big(|f(t,s,0,0)|^2+|g(t,s,0,0)|^2\Big)dsdt\\
&+&\Gamma_{S,T}\Big[\mathbb{E}\int_{S}^{T}|y(t)|^2dt+
\mathbb{E}\int_{S}^{T}\int_{S}^{T}|z(t,s)|^2dsdt\Big],
\end{eqnarray*} where $\Gamma_{S,T}=\max\left\{9C(T-S)\big[(T-S)+1\big];
9\big[(T-S)C+\alpha\big]\right\}$. \\ Hence,
\begin{eqnarray*}
&&||(Y,Z)||_{\mathcal{H}^{2}(\mathcal{D}_{S,T})}^{2}
\\ &\leq& 9(T-S)\mathbb{E}|\xi|^2+9\big[(T-S)\vee
1\big]\mathbb{E}\int_{S}^{T}\int_{S}^{T}\Big(|f(t,s,0,0)|^2+|g(t,s,0,0)|^2\Big)dsdt\\
&+&\Gamma_{S,T}||(y,z)||_{\mathcal{H}^{2}(\mathcal{D}_{S,T})}^{2}.
\end{eqnarray*}

\medskip
\noindent Let us consider the map
$\Theta:\mathcal{H}^{2}(\mathcal{D}_{S,T}) \rightarrow
\mathcal{H}^{2}(\mathcal{D}_{S,T})$ defined by
\begin{eqnarray}
\Theta(y,z)=(Y,Z), \ \ \ \ \forall\ (y,z)\in
\mathcal{H}^{2}(\mathcal{D}_{S,T}),
\end{eqnarray} where $(Y,Z) \in\mathcal{H}^{2}(\mathcal{D}_{S,T})$
is the adapted solution to BDSVIE \eqref{eq5}.

\medskip
\noindent The map $\Theta$, as defined, is a contraction when
$T-S>0$ is small.

\medskip
\noindent Indeed, let $(y',z')\in
\mathcal{H}^{2}(\mathcal{D}_{S,T})$ and $(Y',Z')=\Theta(y',z')$ the
corresponding solution to BDSVIE \eqref{eq5} on $[S,T]$.

\medskip
\noindent Let now define:

\medskip
\hspace{2.5cm}$ \bar{Y}= Y-Y', \ \  \bar{Z}= Z-Z' $ and $ \bar{y}=
y-y', \ \  \bar{z}= z-z'. $

\medskip
\noindent Then,
\begin{eqnarray}\label{4iq}
\bar{Y}_t=\int_{t}^{T}F(t,s,\bar{y}(s),\bar{z}(t,s))ds+
\int_{t}^{T}G(t,s,\bar{y}(s),\bar{z}(t,s))
\overleftarrow{dB_s}-\int_{t}^{T}\bar{z}(t,s)dW_{s},
\end{eqnarray}
where  $F$ and $G$ are defined by
$$F(t,s,u,v)= f(t,s,u+y'(t),v+z'(t,s))-f(t,s,y'(t),z'(t,s))$$
$$G(t,s,u,v)= g(t,s,u+y'(t),v+z'(t,s))-g(t,s,y'(t),z'(t,s)).$$
It is easy to check that $F$ and $G$ verify hypotheses $(H_1)$ and
$(H_2)$ with

\medskip
$F(t,s,0,0)=0$ and $g(t,s,0,0)=0$, for any $(t,s)\in \mathcal{D}.$

\medskip
\noindent Therefore, by lemma \ref{l1}, $(\bar{Y},\bar{Z})\in
\mathcal{M}^{2}(S,T,\mathbb{R}^{k})\times
\mathcal{M}^{2}(\mathcal{D}_{S,T},\mathbb{R}^{k \times d})$ and
\begin{eqnarray*}
\mathbb{E}\int_{S}^{T}|\bar{Y}(t)|^2dt+
\mathbb{E}\int_{S}^{T}\int_{S}^{T}|\bar{Z}(t,s)|^2dsdt \leq
\Gamma_{S,T}\Big[\mathbb{E}\int_{S}^{T}|\bar{y}(t)|^2dt+
\mathbb{E}\int_{S}^{T}\int_{S}^{T}|\bar{z}(t,s)|^2dsdt\Big].
\end{eqnarray*} Hence,
\begin{eqnarray*}
||(\bar{Y},\bar{Z})||_{\mathcal{H}^{2}(\mathcal{D}_{S,T})}^{2}
 \leq \Gamma_{S,T}||(\bar{y},\bar{z})||_{\mathcal{H}^{2}(\mathcal{D}_{S,T})}^{2}.
\end{eqnarray*}
Consequently,
\begin{eqnarray}
||\Theta(y,z)-\Theta(y',z')||_{\mathcal{H}^{2}(\mathcal{D}_{S,T})}^{2}\nonumber
\leq
\Gamma_{S,T}||(y,z)-(y',z')||_{\mathcal{H}^{2}(\mathcal{D}_{S,T})}^{2}\nonumber,
\end{eqnarray} for any $(y,z),\ (y',z')\in \mathcal{H}^{2}(\mathcal{D}_{S,T})$.
%\begin{eqnarray}\label{rpp}
%&&||\Theta(y,z)-\Theta(y',z')||_{\mathcal{H}^{2}(\mathcal{D}_{S,T})}^{2}\nonumber
%\\
%\vspace{2cm} &&\equiv
%||(Y,Z)-(Y',Z')||_{\mathcal{H}^{2}(\mathcal{D}_{S,T})}^{2}\nonumber\\&&\equiv
%\mathbb{E}\left\{\int_{S}^{T}|Y(t)-Y'(t)|^2dt+\int_{S}^{T}\int_{t}^{T}|Z(t,s)-Z'(t,s)|^2dsdt\right\}
%\\
%&\leq &
%\Gamma_{S,T}\mathbb{E}\left\{\int_{S}^{T}|y(s)-y'(s)|^2ds+
%\int_{S}^{T}\int_{t}^{T}|z(t,s)-z'(t,s)|^2dsdt\right\}\nonumber\\
%&\leq &
%\Gamma_{S,T}||(y,z)-(y',z')||_{\mathcal{H}^{2}(\mathcal{D}_{S,T})}^{2}\nonumber.
%\end{eqnarray}

\medskip
\noindent Thus, the map $\Theta:\mathcal{H}^{2}(\mathcal{D}_{S,T})
\rightarrow \mathcal{H}^{2}(\mathcal{D}_{S,T})$ is a contraction
when $S\in[0,T]$ is chosen such that $\Gamma_{S,T}<1$. Hence,
$\Theta$ admits a unique fixed point
$(Y,Z)\in\mathcal{H}^{2}(\mathcal{D}_{S,T})$ which is the unique
adapted solution of BDSVIE \eqref{v2a} on $[S,T]$.

\medskip
To end the proof, let $S'\in [0,S]$.\\ From above, we have the
existence of a unique adapted solution on $[S,T]$. Therefore, there
exists a unique $Y(S)$. \\ Now, for any $t\in[S',S]$, let us
consider the following equation with terminal condition $Y(S)$:
\begin{eqnarray}\label{eqf}
y(t)=Y(S)+\int_{t}^{S}f(t,s,y(s),z(t,s))ds+
\int_{t}^{S}g(t,s,y(s),z(t,s))dB_{s}-\int_{t}^{S}z(t,s)dW_{s}.\ \ \
\end{eqnarray}
Using the same procedure as above, we conclude that the equation
\eqref{eqf} admits a unique adapted solution
$(y,z)\in\mathcal{H}^{2}(\mathcal{D}_{S',S})$ on $[S',S]$.
Therefore, we can deduce by induction, the existence and uniqueness
of an adapted solution in $\mathcal{H}^{2}(\mathcal{D})$ to BDSVIE
\eqref{v2a} on $[0,T]$.
\end{proof}

\end{document}